\documentclass[12pt]{amsart}
\usepackage[colorlinks=true,pagebackref,hyperindex,citecolor=green,linkcolor=red]{hyperref}
\usepackage{fullpage}
\usepackage{setspace}
\usepackage{amsmath}
\usepackage{amsfonts}
\usepackage{amssymb}
\usepackage{bm}
\usepackage{mathrsfs}
\usepackage{mathabx} 
\usepackage{color}
\usepackage[all]{xy}
\usepackage{graphicx}
\usepackage{booktabs}

\usepackage{pstricks}
\usepackage{pst-plot}
\usepackage{pst-math}

\usepackage{verbatim} 

\usepackage{mdwlist}

  \numberwithin{equation}{section}
 
  \theoremstyle{definition}
  \newtheorem{Theorem}{Theorem}[section]
 
 \newtheorem{Lemma}[Theorem]{Lemma}
 
 \newtheorem{Definition}[Theorem]{Definition}
  \newtheorem{Remark}[Theorem]{Remark}
  \newtheorem{Algorithm}[Theorem]{Algorithm}

 \newtheorem{Example}[Theorem]{Example}
 \newtheorem{Notation}[Theorem]{Notation}

\newtheorem*{Problem}{Problem}
\newtheorem*{IExample}{Example}

\newtheorem*{TheoremPt1}{Theorem \ref{BinomialMainTheorem: T}: Part I}  
\newtheorem*{TheoremPt2}{Theorem \ref{BinomialMainTheorem: T}: Part II}  
\newtheorem*{TheoremPt3}{Theorem \ref{BinomialMainTheorem: T}: Part III}

\newtheorem*{BinomialTechnicalLemma1}{Lemma \ref{BinomialTechnicalLemma1: L}}
\newtheorem*{BinomialTechnicalLemma2}{Lemma \ref{BinomialTechnicalLemma2: L}}

\newcommand{\up}[1]{\left\lceil #1 \right\rceil}

\newcommand{\mfrak}[1]{\mathfrak{#1}}

\newcommand{\Spec}{\operatorname{Spec}}

\newcommand{\Support}{\operatorname{Supp}}

\renewcommand{\P}{\boldsymbol{P}}

\renewcommand{\k}{\boldsymbol{k}}
\newcommand{\s}{\boldsymbol{s}}

\renewcommand{\u}{\boldsymbol{u}}

\newcommand{\x}{\boldsymbol{x}}

\newcommand{\m}{\mfrak{m}}

\newcommand{\0}{\boldsymbol{0}}

\renewcommand{\bold}[1]{\mathchoice{\hbox{\boldmath $\displaystyle #1$}}
        {\hbox{\boldmath $\textstyle #1$}}
        {\hbox{\boldmath $\scriptstyle #1$}}
        {\hbox{\boldmath $\scriptscriptstyle #1$}}}

 \newcounter{algorithmcounter}
 \newenvironment{algorithm}
{\begin{list}{\textbf{Step \arabic{algorithmcounter}:}}{\addtolength{\leftmargin}{-.3in}}
\usecounter{algorithmcounter}
  \setlength{\itemsep}{1pt}
  \setlength{\parskip}{1pt}
  \setlength{\parsep}{1pt}}
{\end{list}}

\newcommand{\tr}[2]{\left \langle {#1} \right \rangle_{#2}}

\newcommand{\tail}[2]{ \left \ldbrack {#1} \right \rdbrack_{#2}}

\newcommand{\bracket}[2]{ {#1}^{[p^{#2}]}}

\newcommand{\error}{\varepsilon}

\newcommand{\base}{\operatorname{base}}

\newcommand{\rref}[1]{(\ref{#1})}
\renewcommand{\bracket}[2]{ {#1}^{[p^{#2}]}}

\newcommand{\set}[1]{ \left\{ \, #1 \, \right\}}

\renewcommand{\(}{\left(}
\renewcommand{\)}{\right)}

\newcommand{\leftexp}[2]{{\vphantom{#2}}^{#1}{#2}}

\newcommand{\ulP}{{\leftexp{\star}{\P}}} 
\newcommand{\lrP}{{\P_{\star}}}

\newcommand{\eeta}{\bold{\eta}}

\newcommand{\ssigma}{\bold{\sigma}}

\newcommand{\aalpha}{\bold{\alpha}}
\newcommand{\bbeta}{\bold{\beta}}

\newcommand{\Interior}{\operatorname{Interior}}

\newcommand{\norm}[1]{ \left | #1 \right | }
\newcommand{\digit}[2]{ #1^{\(#2\)} }

\newcommand{\Lattice}[2]{ \frac{1}{#1} \cdot \mathbb{N}^{#2} }

\renewcommand{\th}{\text{th}}

\renewcommand{\exp}[2]{ {#1}^{#2} }

\newcommand{\N}{\mathbb{N}}

\renewcommand{\a}{\bold{a}}
\renewcommand{\b}{\bold{b}}
\newcommand{\sM}{\mathscr{M}}
\newcommand{\dotp}{\bullet} 
\newcommand{\EM}{\bold{\operatorname{E}}}
\newcommand{\vone}{\bold{1}}
\newcommand{\vleq}{\preccurlyeq}
\newcommand{\vl}{\prec}

\newcommand{\vgeq}{\succcurlyeq}

\newcommand{\fpt}[1]{\bold{\operatorname{fpt}}_{\m}(#1)}
\newcommand{\localfpt}[2]{\bold{\operatorname{fpt}}_{#1}(#2)}
\newcommand{\lct}[1]{\bold{\operatorname{lct}}_{\0}(#1)}

\newcommand{\dee}{d}
\renewcommand{\ll}{\bold{\lambda}}

\renewcommand{\gg}{\bold{\gamma}}

\newcommand{\LIP}{P^{\circ}_{\text{lower}}}

\renewcommand{\L}{\mathbb{L}}

\newcommand{\MTP}{Main Theorem} 

\begin{document}


\title
 {$F$-pure thresholds of binomial hypersurfaces}
\author{ Daniel J. Hern\'andez }
\thanks{The author was partially supported by the National Science Foundation RTG grant number 0502170 at the University of Michigan.}

\begin{abstract} 
We use estimates given in \cite{Polynomials} to deduce a formula for the $F$-pure threshold of a binomial hypersurface over a field of characteristic $p>0$.  These formulas are given in terms of the associated splitting polytope, and remain valid over any characteristic.
\end{abstract}

\maketitle

\section{Introduction}
Let $f \in \L[x_1, \cdots, x_m]$ be a polynomial over a field $\L$ of characteristic $p>0$ vanishing at the origin, so that $f \in \m :=(x_1, \cdots, x_m)$. The $F$-pure threshold of $f$, denoted $\fpt{f}$, is an invariant that measures the singularities of $f$ near $\0$, and is defined using properties of the Frobenius morphism.  Remarkably, the $F$-pure threshold is closely related to the \emph{log canonical threshold}, an invariant of singularities defined over $\mathbb{C}$.  Indeed, if $f$ has rational coefficients and $f(\0) = 0$, then one may compute $\lct{f}$, the log canonical threshold of $f$.  However, one may also reduce the coefficients of $f$ modulo $p$ to obtain a family of models $f_p$ over $\mathbb{F}_p$ with $f_p(\0) = 0$, and we have the following relation:  $\lim_{p \to \infty} \fpt{f_p} = \lct{f}$ \cite[Theorem 3.5]{MTW2005}.  Furthermore, it is conjectured that $\fpt{f_p} = \lct{f}$ for infinitely many $p$.  This motivates the following problem.

\begin{Problem}
Given a polynomial $f$ over $\mathbb{Q}$ with $f(\0) = 0$, compute the function \[ \bold{\operatorname{fpt}}: \Spec \mathbb{Z} \to \mathbb{R} \text{ defined by }p \mapsto \fpt{f_p}. \]
\end{Problem}

\begin{IExample}
Let $f \in \mathbb{Q}[x, y, z]$ be a polynomial with $f(\0) = \0$ having an isolated singularity at $\0$, so that $f$ defines an elliptic curve $E \subseteq \mathbb{P}^2_{\mathbb{Q}}$.  It follows from recent work of B. Bhatt (and a generalization by Bhatt and A. Singh) that \[ \bold{\operatorname{fpt}}: \Spec \mathbb{Z} \to \mathbb{R} \text{ is given by } p \mapsto \begin{cases} 1 & \text{ if $E$ is not supersingular at $p$} \\  1 - \frac{1}{p} & \text{ otherwise.} \end{cases} \]
\end{IExample}

Formulas for $\bold{\operatorname{fpt}}$ are rare.  Besides this example, the only such formulas are those for the polynomials $x^2 + y^3, x^2 + y^7$, and $x^5 + y^4 + x^3 y^2$ appearing in \cite{MTW2005}.  However, in the recent preprint \cite{Diagonals}, the author has computed $\bold{\operatorname{fpt}}$ for all \emph{diagonal} polynomials.  

The main result of this note is Theorem \ref{BinomialMainTheorem: T}, which leads to Algorithm \ref{BinomialAlgorithm}, an algorithm for computing the $F$-pure threshold of an arbitrary \emph{binomial} in any (prime) characteristic.  The statement of Theorem \ref{BinomialMainTheorem: T} is technical, so we omit it here and refer the reader to Examples \ref{Comp1: E}, \ref{Comp2: E}, and \ref{Comp3: E} of how Theorem \ref{BinomialMainTheorem: T} is used to compute $F$-pure thresholds. The formulas for $F$-pure thresholds given by Theorem \ref{BinomialMainTheorem: T} are in terms of the geometry of the associated \emph{splitting polytope}; see Definition \ref{SplittingPolytopeDefinition}.  The splitting polytope associated to a binomial is a rational polytope contained in $[0,1]^2$, and was previously used in \cite{TakShi2009} to compute the log canonical threshold of certain binomial \emph{ideals}.

\subsection*{Acknowledgements}
This article is part of my Ph.D. thesis, which was completed at the University of Michigan under the direction of Karen Smith.  I would like to thank Karen, Emily Witt, and Michael Von Korff for participating in many useful discussions.

\section{On base $p$ expansions}

\begin{Definition}
\label{ExpansionDefiniton}  Let $\alpha \in (0,1]$, and let $p$ be a prime number.  There exist a unique collection of integers $\digit{\alpha}{d}$ such that $0 \leq \digit{\alpha}{d} \leq p-1, \alpha = \sum_{d \geq 1} \frac{\digit{\alpha}{d}}{p^d}$, and such that $\digit{\alpha}{d}$ is not eventually zero as a function of $d$.  The integers $\digit{\alpha}{d}$ are called the \emph{digits} of $\alpha$ (in base $p$), and the expression $\alpha = \sum_{d \geq 1} \frac{\digit{\alpha}{d}}{p^d}$ is called the \emph{non-terminating} (base $p$) \emph{expansion of $\alpha$}.
\end{Definition}

\begin{Definition}
\label{TruncationDefinition}

Let $\alpha \in  [0,1]$, and let $p$ be a prime number.  We use $\tr{\alpha}{e}$ to denote $\sum_{d=1}^e \frac{ \digit{\alpha}{d}}{p^d}$, the \emph{$e^{\th}$ truncation} of $\alpha$ (in base $p$), and $\tail{\alpha}{e}$ to denote $\alpha - \tr{\alpha}{e} = \sum_{d > e} \frac{ \tr{\alpha}{d}}{p^d}$, the \emph{$e^{\th}$ tail} of $\alpha$ (in base $p$).  We adopt the convention that $\tr{\alpha}{0} = \tr{0}{e} = \tail{0}{e} = 0$.
\end{Definition}

\begin{Lemma}
\label{Truncation: L} For $\alpha, \beta \in [0,1]$, we have the following:
\begin{enumerate}
\item $\tr{\alpha}{e} \in \frac{1}{p^e} \cdot \mathbb{N}$ and $\tr{\alpha}{e} < \alpha$.
\item $0<\tail{\alpha}{e} \leq \frac{1}{p^e}$, with equality if and only if $\alpha \in \Lattice{p^e}{2}$. 
\item $\alpha \leq \beta$ if and only if $\tr{\alpha}{e} \leq \tr{\beta}{e}$ for all $e \geq 1$.
\item If $\beta \in [0,1] \cap \frac{1}{p^e} \cdot \mathbb{N}$ and $\alpha > \beta$, then $\tr{\alpha}{e} \geq \beta$.
\end{enumerate}
\end{Lemma}

\begin{proof}
The first two assertions follow directly from the definitions.  We now prove the last point, leaving the third for the reader.  As $\beta < \alpha$, we have that \[ p^e \beta < p^e \alpha = p^e \tr{\alpha}{e} + p^e \tail{\alpha}{e} \leq p^e \tail{\alpha}{e} + 1,\] where we have used that $\tail{\alpha}{e} \leq \frac{1}{p^e}$.  Thus, $p^e \beta < p^e \tr{\alpha}{e} + 1$, and as both sides are integers, we conclude that $p^e \beta \leq p^e \tr{\alpha}{e}$.
\end{proof}

\begin{Definition}
\label{Carrying: D}
Let $(\alpha, \beta) \in [0,1]^2$, and let $p$ be a prime number.  We say  that $e^{\th}$ digits of $\alpha$ and $\beta$ add without carrying if $\digit{\alpha}{e} + \digit{\beta}{e} \leq p-1$, and we say that $\alpha$ and $\beta$ \emph{add without carrying} (in base $p$)  if the $e^{\th}$ digits of $\alpha$ and $\beta$ add without carrying for all $e$. We  define the corresponding conditions for integers in the obvious way.
\end{Definition}

\begin{Remark}
\label{Carrying: R}  It follows from the definitions that $\alpha$ and $\beta$ add without carrying (in base $p$) if and only if the integers $p^e \tr{\alpha}{e}$ and $p^e \tr{\beta}{e}$ add without carrying for all $e \geq 1$.
\end{Remark}

\begin{Example} 
\label{ConstantExpansionCarrying: E}
Let $\alpha \in [0,1]$.  If $(p-1) \cdot \alpha \in \mathbb{N}$, then multiplying the non-terminating base $p$ expansion $1 = \sum_{e \geq 1} \frac{p-1}{p^e}$ by $\alpha$ shows that $\digit{\alpha}{e} = (p-1) \cdot \alpha$ for all $e \geq 1$.  In particular, if $(\alpha, \beta) \in [0,1]^2$ and $(p-1) \cdot (\alpha, \beta) \in \mathbb{N}^2$, then $\alpha$ and $\beta$ add without carrying (in base $p$) if and only if $\alpha + \beta \leq 1$.  
\end{Example}

\begin{Lemma}\cite{Lucas} 
\label{Lucas: L}  
 Let $k_1, k_2 \in \mathbb{N}$, and set $N = k_1 + k_2$. Then, the binomial coefficient $\binom{N}{k_1, k_2} : = \frac{ N ! }{k_1 ! k_2!} \neq 0 \bmod p$ if and only if $k_1$ and $k_2$ add without carrying (in base $p$).
\end{Lemma}

\begin{Lemma}
\label{Carrying: L}
Consider $ \( \alpha, \beta \) \in [0,1]^2$, and suppose that $\alpha + \beta \leq 1$.  If $\digit{\alpha}{L+1} + \digit{\beta}{L+1} \geq p$, then $\tr{\alpha}{L} + \tr{\beta}{L} + \frac{1}{p^L} = \tr{\alpha + \beta}{L}$.
\end{Lemma}

\begin{proof}

By definition,
\begin{equation} \label{twodigitsadding1: e} \tr{\alpha + \beta}{L} + \tail{\alpha + \beta}{L} = \alpha + \beta = \tr{\alpha}{L} + \tr{\beta}{L} + \tail{\alpha}{L} + \tail{\beta}{L}.\end{equation}

  As the digits appearing in a base $p$ expansion are less than or equal to $p-1$, 
\begin{align*}
\frac{1}{p^L} < \tail{\alpha}{L} + \tail{\beta}{L}  & = \frac{1}{p^L} + \frac{\digit{\alpha}{L+1} + \digit{\beta}{L+1} - p}{p^{L+1}} + \tail{\alpha}{L+1} + \tail{\beta}{L+1} \\ 
  & \leq \frac{1}{p^L} + \frac{2p-2 - p}{p^{L+1}} + \frac{1}{p^{L+1}} + \frac{1}{p^{L+1}} = \frac{2}{p^L},
\end{align*}
\noindent so that $\up{p^L \tail{\alpha}{L} + p^L \tail{\beta}{L}} = 2$.  Thus, multiplying \eqref{twodigitsadding1: e} by $p^L$ and rounding up, keeping in mind that $ 0 < \tail{\alpha + \beta}{L} \leq \frac{1}{p^L}$, shows that $p^L \tr{\alpha + \beta}{L} + 1 = p^L \tr{\alpha}{L} + p^L \tr{\beta}{L} + 2$.
\end{proof}

\section{$F$-pure thresholds and splitting polytopes}

\subsection{$F$-pure thresholds of polynomials}

Recall that a field $\L$ of positive characteristic is said to be \emph{$F$-finite} if $[\L:\L^p] < \infty$.  We will assume that all fields of positive characteristic are $F$-finite.  Fix such a field, and let $f \in \L[x_1, \cdots, x_m]$ be a polynomial vanishing at the origin, so that $f \in \m :=(x_1, \cdots, x_m)$.  For every $e \geq 1$, let $\bracket{\m}{e}$ denote $(x_1^{p^e}, \cdots, x_m^{p^e})$, the \emph{$e^{\th}$ Frobenius power of $\m$}.  


\begin{Definition} \cite{MTW2005, BMS2008} 
\label{FPT: D}
The limit $\fpt{f}:= \lim_{e \to \infty} \frac{1}{p^e} \cdot \max \set{ l : f^l \notin \bracket{\m}{e}}$ exists and is contained in $(0,1] \cap \mathbb{Q}$.  We call this number the \emph{$F$-pure threshold} of $f$ at $\m$.  
\end{Definition}

Lemma \ref{FPTTruncation: L} allows us to recover the terms $\max \set{ l : f^l \notin \bracket{\m}{e}}$ from their limiting value.

\begin{Lemma}
\label{FPTTruncation: L}  Let $f \in K[x_1, \cdots, x_m]$ be a polynomial vanishing at $\0$.  Then 
\[ \tr{\fpt{f}}{e} = \frac{ \max \set{ l : f^l \notin \bracket{\m}{e}}}{p^e}.\] 
\end{Lemma}

\begin{proof}
This is a restatement of \cite{MTW2005} in the language of base $p$ expansions.  See \cite{Polynomials} for a generalization.
\end{proof}

\begin{Lemma}
\label{FPTDifferentVariables: L}  Let $\m$ (respectively, $\mathfrak{n}$) denote ideal generated by the variables in $\L[x_1, \cdots, x_m]$ (respectively, $\L[y_1, \cdots, y_n]$).  If $f \in \m$ and $g \in \mathfrak{n}$, then $\localfpt{\m + \mathfrak{n}}{fg} = \min \{ \localfpt{\m}{f} , \localfpt{\mathfrak{n}}{g} \}$.
\end{Lemma} 

\begin{proof}  Let $\lambda_1 = \localfpt{\m}{f}$ and $\lambda_2 = \localfpt{\mathfrak{n}}{g}$, and suppose that $\lambda_1 \leq \lambda_2$. We now show that $\localfpt{\m +  \mathfrak{n}}{fg} = \lambda_1$.  As $\lambda_1 \leq \lambda_2$,  it follows from Lemma \ref{Truncation: L} that $\tr{\lambda_1}{e} \leq \tr{\lambda_2}{e}$ for all $e \geq 1$, and so by Lemma \ref{FPTTruncation: L} we have that $f^{p^e \tr{\lambda_1}{e}} \notin \bracket{\m}{e}$ and $g^{p^e \tr{\lambda_1}{e}} \notin \bracket{\mathfrak{n}}{e}$.  As $f$ and $g$ are in different sets of variables, it follows that $(fg)^{p^e \tr{\lambda_1}{e}} \notin \bracket{\( \m + \mathfrak{n} \) }{e} $.   By Lemma \ref{FPTTruncation: L}, we know that $f^{p^e \tr{\lambda_1}{e} + 1} \in \bracket{\m}{e}$, and so $(fg)^{p^e \tr{\lambda_1}{e} + 1} \in \bracket{\m}{e} \subseteq \bracket{ \( \m + \mathfrak{n} \) }{e}$.  We conclude that \[ p^e \tr{\lambda_1}{e} = \max \{ l : (fg)^{l} \notin \bracket{ \( \m + \mathfrak{n} \) }{e} \} = p^e \tr{\localfpt{\m + \mathfrak{n}}{fg}}{e},\] where the last inequality holds by \ref{FPTTruncation: L}.  The claim follows by letting $e \to \infty$.
\end{proof}

\subsection{Splitting polytopes dimension two}

We begin by reviewing some standard conventions from convex geometry. Recall that a set $\mathscr{P} \subseteq \mathbb{R}^2$ is called a \emph{rational polyhedron} if there exist finitely many linear forms $L_1, \cdots, L_d \in \mathbb{Q}[t_1,t_2]$ such that \[\mathscr{P} = \set{ \s \in \mathbb{R}^2:  L_i (\s) \leq 1 \text{ for all } 1 \leq i \leq d}.\]  A bounded rational polydedron is called a \emph{rational polytope}.  If $L \in \mathbb{Q}[t_1, t_2]$ and $\beta \in \mathbb{Q}$, the set $H^{\beta}_L:=\set{ \s : L(\s) \leq \beta}$ is called a \emph{supporting halfspace} of $\mathscr{P}$ if $\mathscr{P} \subseteq H^{\beta}_L$ and $L^{-1}(\beta) \cap \mathscr{P}$ is non-empty.  In this case, we call the set $L^{-1}(\beta) \cap \mathscr{P}$ a \emph{face} of $\mathscr{P}$.  An element $\bold{\text{v}}$ is called a \emph{vertex} of $\mathscr{P}$ if $\set{ \bold{\text{v}} }$ is a face of $\mathscr{P}$.

\begin{Notation}
\label{UniNotation: N}
Let $\sM = \set{ \x^{\a}, \x^{\b}}$ denote a collection of distinct monomials in the variables $x_1, \cdots, x_m$ such that every variable appears in either $\x^{\a}$ or $\x^{\b}$.  We  use $\vl$ and $\vleq$ to denote componentwise inequality in $\mathbb{R}^m$,  $\dotp$ to denote the standard dot product on $\mathbb{R}^2$, and $\vone_m$ to denote the element $(1, \cdots, 1) \in  \mathbb{N}^m$.
\end{Notation}

\begin{Definition} 
\label{SplittingPolytopeDefinition}
Let $\EM$ denote the $m \times 2$ matrix $( \a \ \b)$.  We call $\EM$ the \emph{splitting matrix} of $\sM$, and $\P = \set{ \s \vgeq \0 : \EM \s \vleq \vone_m}$ the \emph{splitting polytope} of $\sM$.  As $\a, \b \in \mathbb{N}^m$, we see that $\P \subseteq [0,1]^2$ a rational polytope.
\end{Definition}

\begin{figure}[h]
\begin{minipage}[h]{2.5in}
\begin{center}
\psset{xunit=900pt, yunit=900pt}
\begin{pspicture}(0,-.01)(.18,.16)
\psaxes[ticks=none]{->}(0,0)(0,0)(.17,.14)[$s_1$,270][$s_2$,180]
\psline(0,0)(0,.1111)(.0357, .1071)(.1,.075)(.1428,0)(0,0)
\psdots[dotsize=6pt](0,0)(0,.1111)(.0357, .1071)(.1,.075)(.1428,0)(0,0)

\uput[180](0,.1111){$\(0,\frac{1}{9}\)$}
\uput[45](.0357, .1071){$\( \frac{1}{28}, \frac{3}{28} \)$}
\uput[45](.1,.075){$\( \frac{1}{10}, \frac{3}{40} \)$}
\uput[55](.1428,0){$\( \frac{1}{7}, 0 \)$}

\end{pspicture}
\end{center}
\end{minipage}
\begin{minipage}[h]{2.5in}
\[\P = \set{ (s_1, s_2) \vgeq \0 : 
\begin{array}{l}
       \hspace{5pt} s_1 + 9 s_2 \leq 1 \\
       4 s_1 + 8 s_2 \leq 1 \\
       7s_1 + 4 s_2 \leq 1
\end{array}
}
\]
\end{minipage}
\caption{The splitting polytope of  $\set{ x y^{4} z^7, x^9 y^8 z^4}$. }
\label{PolytopeFig: f}
\end{figure}
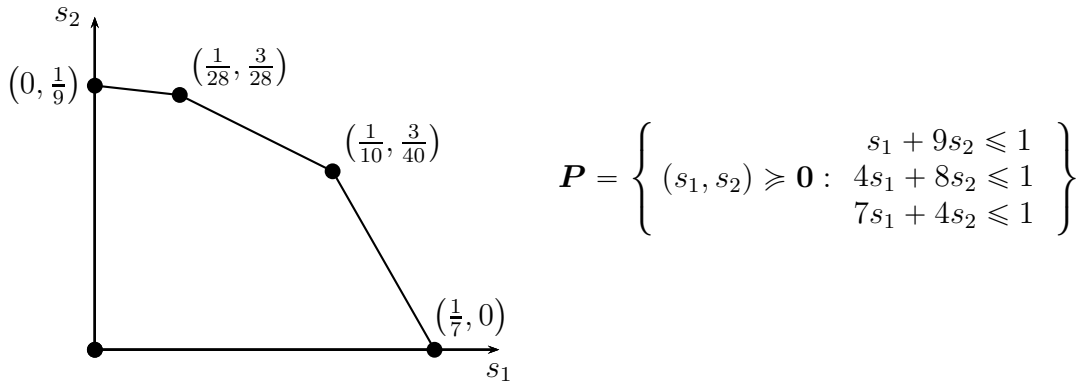


\begin{Definition}
\label{MaximalPoint: D}
We call the set $\set { \s \in \P : \EM \s \vl \vone_m}$ the \emph{lower interior} of $\P$, and we denote it by $\LIP$.  An element $\eeta  \in \P$ is called a \emph{maximal point of $\P$} if $|\eeta| = \max \left \{ | \s | : \s \in \P \right \}$.
\end{Definition}

Set $\alpha = \max \set{ | \s | : \s \in \ P}$.  By definition, $H_{\max}:=\set{ \s : \( \frac{1}{\alpha} , \frac{1}{\alpha} \) \dotp \s \leq 1 }$ is a supporting halfspace of $\P$, but is not one the defining halfspaces of $\P$ unless there exists a row $(a_i, b_i)$ of $\EM$ with $a_i = b_i = \frac{1}{\alpha}$.  Whenever $a_i \neq b_i$ for all rows $(a_i, b_i)$, it follows that the face determined by $H_{\max}$ must be a vertex of $\P$.  We record this observation below.

\begin{Lemma}  
\label{UniquePointCriteria: L}
If $\EM$ has no constant rows, then $\P$ has a unique maximal point.
\end{Lemma}





\begin{Definition}  
\label{SubPolytopes: D}
Suppose that $\P$ has a unique maximal point $\eeta = \( \eta_1, \eta_2 \) \in \P$, and set \[ \ulP: =\set{ \s \in \P : s_2 \geq \eta_2}
\text{ and } \lrP:= \set{ \s \in \P : s_1 \geq \eta_1}.\]  
Note that this gives a decomposition $\P = \ulP 
 \cup \lrP \cup \set{ \s  \in \P : \s \vleq \eeta}$.
\end{Definition}

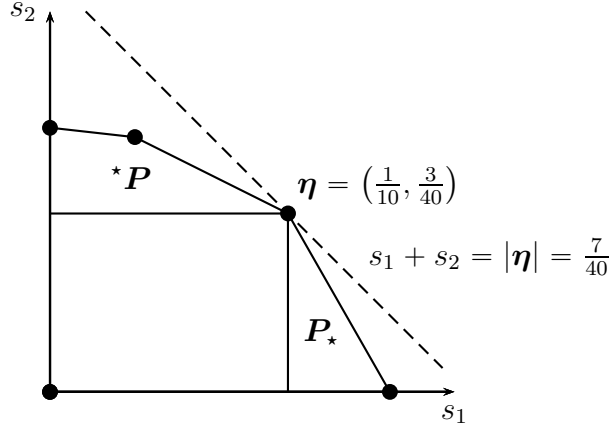
\begin{figure}[h]
\begin{center}
\psset{xunit=900pt, yunit=900pt}
\begin{pspicture}(0,-.01)(.18,.16)
\psaxes[ticks=none]{->}(0,0)(0,0)(.17,.16)[$s_1$,270][$s_2$,180]
\psline(0,0)(0,.1111)(.0357, .1071)(.1,.075)(.1428,0)(0,0)
\psplot[linestyle=dashed]{.015}{.165}{x neg .175 add}
\psline(0,.075)(.1,.075)
\psline(.1,0)(.1,.075)
\psdots[dotsize=6pt](0,0)(0,.1111)(.0357, .1071)(.1,.075)(.1428,0)(0,0)


\uput[0](.02,.09){$\ulP$}
\uput[0](.101,.025){$\lrP$}

\uput[45](.1,.075){$\eeta= \( \frac{1}{10}, \frac{3}{40} \)$}

\uput[45](.13,.045){$s_1 + s_2 = | \eeta | = \frac{7}{40} $}
\end{pspicture}
\end{center}

\caption{The decomposition of $\P$ in Figure \ref{PolytopeFig: f}}
\label{Decomposition: f}
\end{figure}

In Figure \ref{Decomposition: f}, we see that the faces of $\P$ with slope less than $-1$ correspond to  faces of $\lrP$, while those with slope between $0$ and $-1$ correspond to faces of $\ulP$.  This observation holds in general, and is the key idea behind the following lemma. 

\begin{Lemma}
\label{PrelimSimplifying: L}
Suppose that  $\P$ contains a unique maximal point $\eeta$.  If $\s \in \ulP$, then $  \s \in \LIP$ if and only if $(a_i, b_i) \dotp \s < 1$ for all $i$ with $b_i  > a_i$.  
\end{Lemma}

\begin{proof}
We may assume that $\s \neq \eeta$, as $\eeta$ does not satisfy either of the conditions. It suffices to show that if $\eeta \neq \s \in \ulP$, then $\s$ automatically satisfies the conditions $(a_i, b_i) \dotp \s < 1$ for all $i$ with $a_i > b_i$.  By means of contradiction, suppose otherwise, so that there exists a row $(a_i, b_i)$ with $a_i > b_i$ such that $(a_i, b_i) \dotp \s \geq 1$.  For simplicity of notation, we denote this row by $(a, b)$ with $a>b$.  As $\eeta \in \P$, 
\begin{equation} \label{prelimsimplify1: e} (a, b) \dotp \eeta \leq 1 \leq (a,b) \dotp \s = (a-b, 0) \dotp \s + (b,b) \dotp \s  < (a-b,0) \dotp \s + (b,b) \dotp \eeta, \end{equation}
\noindent where we have used in \eqref{prelimsimplify1: e}, keeping in mind that $\s \neq \eeta$, that \[ (b,b) \dotp \s = b \cdot | \s | <  b \cdot |\eeta|  = (b,b) \dotp \eeta.\]  It follows from \eqref{prelimsimplify1: e} that $(a-b,0) \dotp \eeta < (a-b, 0) \dotp \s$, and as $a  > b$, we conclude that $\eta_1 < s_1$.  However, substituting this into the inequality $| \s | < | \eeta |$ implies that $s_2 < \eta_2$, contradicting the fact that $\s \in \ulP$.
\end{proof}

The following lemma will be especially important in the proof of Theorem \ref{BinomialMainTheorem: T}.  

\begin{Lemma}
\label{GrandSimplifying: L}
Suppose that $\P$ contains a unique maximal point $\eeta$, and let  $\delta \geq 0$.  There exists $\gg = (\gamma_1, \gamma_2) \in \LIP$ with $\gamma_2 \geq \tr{\eta_2}{\dee} + \frac{1}{p^{\dee}}$ and $| \gg | = \tr{\eta_1}{d} + \tr{\eta_2}{d} + \frac{1}{p^d} + \delta$ if and only if $\tr{\eeta}{d} + \(\delta,  \frac{1}{p^d}  \) \in \LIP$.  
\end{Lemma}

\begin{proof}  Set $\ll = (\lambda_1, \lambda_2) : = \tr{\eeta}{d} + \( \delta, \frac{1}{p^d} \)$.  If  $\ll \in \LIP$, one may take $\gg = \ll$.

Next, suppose there exists $\gg$ satisfies the given properties, but that $\ll \notin \LIP$.  As $\ll \in \ulP$, Lemma \ref{PrelimSimplifying: L} implies there exists a row $(a_i, b_i)$ of $E$ such that $b_i > a_i$ and $(a_i, b_i) \dotp \ll \geq 1$.  To simplify notation, we denote this row by $(a,b)$ with $b>a$.  Then 
\begin{equation} \label{grandsimplification2: e} (a,a) \dotp \gg + (0, b-a) \dotp \gg  = (a,b) \dotp \gg < 1 \leq (a,b) \dotp \ll = (a,a) \dotp \ll + (0, b-a) \dotp \ll.\end{equation}

As $| \gg | = | \ll |$, we may cancel the summands $(a,a) \dotp \gg$ and $(a,a) \dotp \ll$ in \eqref{grandsimplification2: e} to obtain $(0, b-a) \dotp \gg < (0, b-a) \dotp \ll$.  As $b > a$, we conclude that $\gamma_2 < \lambda_2 = \tr{\eta_2}{\dee} + \frac{1}{p^{\dee}}$, a contradiction.
\end{proof}

\subsection{Connections with $F$-pure thresholds of binomials}

\begin{Definition}
If $g \in \mathbb{L}[x_1, \cdots, x_m]$ for some field $\mathbb{L}$, we use $\Support(g)$ to denote the unique collection of monomials $\mathscr{N}$ such that $g$ is a $\mathbb{L}^{\ast}$-linear combination of the elements of $\mathscr{N}$.
\end{Definition}

\begin{Lemma}
\label{StandardPolytope: L}
Suppose that $\P$ contains a unique maximal point $\eeta$ with $| \eeta | > 1$.  Then, after possibly exchanging its columns, $\EM$ contains  rows of the form $(1, 0)$ and $(0, n)$ for some $n \geq 1$.  In particular, $\fpt{f} = 1$ for all polynomials with $\Support(f) = \sM$.
\end{Lemma}

\begin{proof}
As $P \subseteq [0,1]^2$, the assumption that $\eta_1 + \eta_2 > 1$ implies that $\eta_1$ and $\eta_2$ are non-zero, so that $\ulP$ and $\lrP$ are non-empty.   By Definition \ref{SubPolytopes: D}, $\eeta$ is a vertex of both $\ulP$ and $\lrP$, so  
 \begin{equation} \label{StandardPolytope: e} (a_i, b_i) \dotp \eeta = (a_j, b_j) \dotp \eeta = 1 \end{equation} for some $(a_i, b_i)$ with $a_i < b_i$ and $(a_j,b_j)$ with $b_j > a_j$.  As $a_i < b_i$, it follows that 
\[ a_i < a_i \cdot | \eeta | = ( a_i, a_i) \dotp \eeta < ( a_i, b_i) \dotp \eeta = 1. \] As $a_i \in \mathbb{N}$, it follows that $a_i = 0$.  Similarly, $b_j = 0$. Substituting these values into \rref{StandardPolytope: e} shows that $\eeta = \( \frac{1}{a_j}, \frac{1}{b_i} \)$, and the equality $\frac{1}{a_j} + \frac{1}{b_i} = | \eeta | > 1$ shows that either $a_j$ or $b_i$ must equal $1$.  Thus, after possibly swapping the columns, we see that $(a_j, 0) = (1,0)$ and $(0, b_i) = (0, n)$ for some $n \geq 1$.  Rename the variables so that the variable corresponding to $(1,0)$ is $y$ and the variable corresponding to $(0,n)$ is $z$.  It follows that  $\sM = \set{ y  \mu_1, z^n  \mu_2}$ for monomials $\mu_1$ and $\mu_2$ containing no powers of $y$ or $z$.  If $f = u_1 y  \mu_1 + u_2 z^n  \mu_2$, the linear factor $y$ implies that  $ \max \set{ l : f^l \notin (y^{p^e}, z^{p^e})} = p^e - 1$.  It follows from Lemma \ref{FPTTruncation: L} that $\localfpt{(y,z)}{f} = 1$.
\end{proof}



\section{$F$-pure thresholds of binomials}

We continue to use $\sM = \set{ \x^{\a}, \x^{\b}}$ to denote a collection of distinct monomials in the variables $x_1, \cdots, x_m$ such that every variable appears in either $\x^{\a}$ or $\x^{\b}$.  All polynomials are assumed to be over an $F$-finite field $\L$ of characteristic $p>0$.

\subsection{Statement of the Main Theorem and an Algorithm}

\begin{Theorem}
\label{BinomialMainTheorem: T}
Let $f$ be a polynomial with $\Support(f) = \sM$.  Suppose $\P$ contains a unique maximal point $\eeta$ with $| \eeta | \leq 1$, and let $L := \sup \{ N : \digit{\eta_1}{e} + \digit{\eta_2}{e} \leq p-1 \text{ for }  0 \leq e \leq N \}$. 
\begin{enumerate}
\item If $L = \infty$, then $\fpt{f} = \norm{\eeta}$.
\suspend{enumerate}
If $L < \infty$, set $\dee: = \max \set{ e \leq L : \digit{\eta_1}{e} + \digit{\eta_2}{e} \leq p-2}$. We will see that $1 \leq \dee \leq L$.
\resume{enumerate}
\item If neither $\tr{\eeta}{\dee} + \(\frac{1}{p^{\dee}}, 0\)$ nor $\tr{\eeta}{\dee} + \( 0, \frac{1}{p^{\dee}} \)$ is in $\LIP$, then $\fpt{f} = \tr{ \norm{ \eeta} }{L}$.  \\ 
\item Otherwise, let $\error = \max \set{ \delta : \tr{\eeta}{\dee} + \( \frac{1}{p^{\dee}}, \delta \) \text{or } \tr{\eeta}{\dee} + \( \delta, \frac{1}{p^{\dee}} \) \text{is in $\P$}}$.  Then \vspace{.05in}
\begin{enumerate}
\item $0 < \error \leq \tail{\norm{\eeta}}{L}$, with equality if and only if either $\eta_1$ or $\eta_2$ is in $\Lattice{p^{\dee}}{}$, and \vspace{.05in}
\item $\fpt{f} = \tr{\norm{\eeta}}{L} + \error$.
\end{enumerate}
\end{enumerate}
\end{Theorem}

We now show how Theorem \ref{BinomialMainTheorem: T} may be used to construct an algorithm that will allow us to compute the $F$-pure threshold at $\m$ of \emph{any} binomial over $K$.

\begin{Algorithm}
\label{BinomialAlgorithm}
Let $g$ be a binomial in $\L[x_1, \cdots, x_m]$.   Our goal is to compute $\fpt{g}$. 
\begin{algorithm}
\item Factor $g= \mu \cdot h$ for some momial $\mu$ and binomial $h$ with the property that no variable appearing in $\mu$ appears in $h$, and so that no variable appears with the same exponent in both supporting monomials of $h$.  By Lemma  \ref{UniquePointCriteria: L}, the polytope $\P$ associated to $\Support(h)$ will contain a unique maximal point $\eeta \in \P$.  
\item  Reorder the variables so that $\mu =x_1^{a_1} \cdots x_d^{a_d}$ .  By Lemma \ref{FPTDifferentVariables: L},  \[ \fpt{g} = \fpt{\mu \cdot h} =  \min \set{ \localfpt{ \( x_1, \cdots, x_d \)}{ \mu }, \localfpt{(x_{d+1}, \cdots, x_m)}{h} }.\]  It is an easy exercise to verify that $\localfpt{ (x_1, \cdots, x_d)}{\mu} = \min \set{ \frac{1}{a_1}, \cdots, \frac{1}{a_d}}$.  
\item If $| \eeta | >1$, it follows from Lemma \ref{StandardPolytope: L} that $\fpt{h}=1$.
\item If $| \eeta | \leq 1$, we may compute $\fpt{h}$ using Theorem \ref{BinomialMainTheorem: T}.
\end{algorithm}
\end{Algorithm}

We now present a series of examples which show Theorem \ref{BinomialMainTheorem: T} in action.

\begin{Example}  
\label{Comp1: E}
Let $\sM = \set{ x^7 y^2, x^5 y^6}$, and choose $f$ with $\Support(f) = \sM$.  It follows that \[ \P = \set{ (s_1, s_2) \in \mathbb{R}^2_{\geq 0} : \begin{array}{l}
       2 s_1 + 6 s_2 \leq 1 \\
       7 s_1 + 5 s_2 \leq 1 \\
\end{array}
}
.\]  
Note that $\P$ contains a unique maximal point $\eeta = \(\frac{1}{32}, \frac{5}{32} \)$ (see Figures \ref{BinomialExample p=43: F} and \ref{BinomialExample p=37: F}).  By Theorem \ref{BinomialMainTheorem: T}, we see that $\fpt{f} = | \eeta | = \frac{3}{16}$ if and only if $\frac{1}{32}$ and $\frac{5}{32}$ add without carrying (in base $p$).  By Example \ref{ConstantExpansionCarrying: E}, if $p \equiv 1 \bmod 32$, (i.e., if $p=97, 193, 257, 353$ or $449$) then $\frac{1}{32}$ and $\frac{5}{32}$ add without carrying.  Note that there are infinitely many such primes by Dirichlet's theorem.  However, there exist $p$ such that $\fpt{f} = \frac{3}{16}$ with $p \not \equiv 1 \bmod 32$.  For example, if $p = 47$, then $\frac{1}{32} = . \overline{ \ 1 \ 22} \ (\base 47) \text{ and } \frac{5}{32} = . \overline{ \ 7 \ 16} \ (\base 47)$.
\end{Example}

\begin{Example}  
\label{Comp2: E}
We now compute $\fpt{f}$ when $p=43$.  As \begin{equation} \label{BinomialComputation1: e} \frac{1}{32} = . \overline{ \ 1 \ 14 \ 33 \ 25 \ 22 \ 36 \ 12 \ 4 } \ (\base 43) \text{ and } \frac{5}{32} = . \overline{ \ 6 \ 30 \ 38 \ 41 \ 28 \ 9 \ 17 \ 20 } \ (\base 43) ,\end{equation} we see that carrying occurs with the second digits, and that $\dee = L = 1$.  By \rref{BinomialComputation1: e}, we see that $\tr{\eeta}{1} = \( \frac{1}{43}, \frac{6}{43} \)$, and Figure \ref{BinomialExample p=43: F} shows that neither $\tr{\eeta}{1} + \( \frac{1}{43}, 0 \)$ nor $\tr{\eeta}{1} + \( 0, \frac{1}{43} \)$ is contained in $\LIP$.   Theorem \ref{BinomialMainTheorem: T} allows us to conclude that $\fpt{f} = \tr{ | \eeta |}{1} = \tr{ \frac{3}{16} }{1} = \frac{8}{43}$.
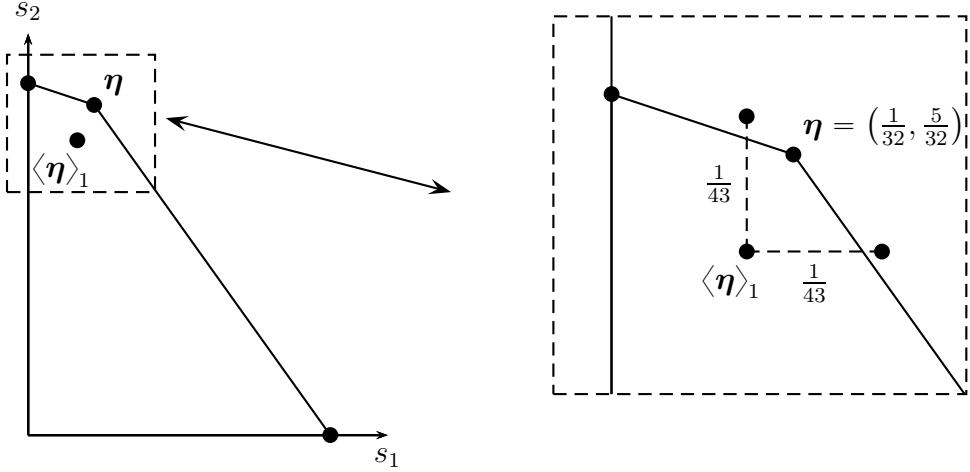
\begin{figure}[h]
\begin{center}
\begin{minipage}[h]{3in}
\psset{xunit=800pt, yunit=800pt}
\begin{pspicture}(0,-.01)(.18,.2)
\psaxes[ticks=none]{->}(0,0)(0,0)(.17,.19)[$s_1$,270][$s_2$,90]
\psline(0,0)(0,.1666)(.03125, .15625)(.14285,0)(0,0)


\uput[45](.03125, .15625){$\eeta$}

\psdots[dotsize=6pt](0,.1666)(.03125, .15625)(.14285,0)
\psdot[dotsize=6pt](.023255,.13953)

\uput[-110](.023255,.13953){$\tr{\eeta}{1}$} 

\psline[linestyle=dashed](-.01,.115)(.06,.115)(.06,.18)(-.01,.18)(-.01,.115)

\psline[arrowsize=6pt]{<->}(.065,.15)(.2,.115)

\end{pspicture}
\end{minipage}
\begin{minipage}[h]{2.2in}
\psset{xunit=2200pt, yunit=2200pt}
\begin{pspicture}(0,-.01)(.06,.065)

\psline(0,0)(0,.0516)(.03125, .04125)(.060714,0)


\psdots[dotsize=6pt](0,.0516)(.03125, .04125)

\uput[45](.03125, .04125){$\eeta = \( \frac{1}{32}, \frac{5}{32} \) $}


\psdot[dotsize=6pt](.023255,.02453)
\uput[-110](.023255,.02453){$\tr{\eeta}{1}$}

\psset{linestyle=dashed, dotsize=6pt}
\psline(.023255,.02453)(.023255, .0477858)
\psline(.023255,.02453)(.0465108, .02453)
\psdots(.023255, .0477858)(.0465108, .02453)

\uput[270](.0348829, .02453){$\frac{1}{43}$}

\uput[180](.023255, .0361579){$\frac{1}{43}$}

\psline[linestyle=dashed](-.01,.0)(.061,.0)(.061,.065)(-.01,.065)(-.01,.0)

\psline[linestyle=solid](0,0)(0,.065)
\end{pspicture}
\end{minipage}

\caption{$p=43$, $\dee = L = 1$, and $\fpt{f} = \tr{\frac{3}{16}}{1} = .8 \ (\base 43)$. }
\label{BinomialExample p=43: F}
\end{center}
\end{figure}
\end{Example}

\begin{Example}
\label{Comp3: E}
We end by computing $\fpt{f}$ when $p=37$.  As \begin{equation} \label{BinomialComputation2: e}  \frac{1}{32} = .\overline{ \ 1 \ 5 \ 28 \ 33 \ 19 \ 24 \ 10 \ 15 } \ (\base 37) \text{ and } \frac{5}{32} = .\overline{ \ 5 \ 28 \ 33 \ 19 \ 24 \ 10 \ 15 \ 1 } \ (\base 37), \end{equation}
we see that the first carry occurs with the third digits, and that $\dee = L = 2$.  We also see from \rref{BinomialComputation2: e} that $\tr{\eeta}{2} = \( \frac{1}{37} + \frac{5}{{37}^2}, \frac{5}{37} + \frac{28}{37^2} \)$.  From Figure \ref{BinomialExample p=37: F}, we see that both $\tr{\eeta}{2} + \( \frac{1}{37^2},0 \)$ and $\tr{\eeta}{2} + \( 0, \frac{1}{37^2} \)$ are contained in $\LIP$.

From Figure \ref{BinomialExample p=37: F}, we also see that $\error = \max \set{ \delta : \tr{\eeta}{2} + \( \frac{1}{{37}^2}, \delta \)  \in \P }$.  Morever, Figure \ref{BinomialExample p=37: F} shows that the point $\tr{\eeta}{2} + \( \frac{1}{37^2}, \error \)$ likes on the hyperlane $7 s_1 + 5 s_2 = 1$, and an easy calculation shows that $\error = \frac{3}{6845} = . 0 \ 0 \ \overline{22 \ 7 \ 14 \ 29 } \ (\base 37)$.  Thus, Theorem \ref{BinomialMainTheorem: T} shows that $\fpt{f} = \tr{ | \eeta | }{2} + \error =  \tr{\frac{3}{16}}{2} + \error = .6 \ 34 \ \overline{22 \ 7 \ 14 \ 29} \ (\base 37)$.


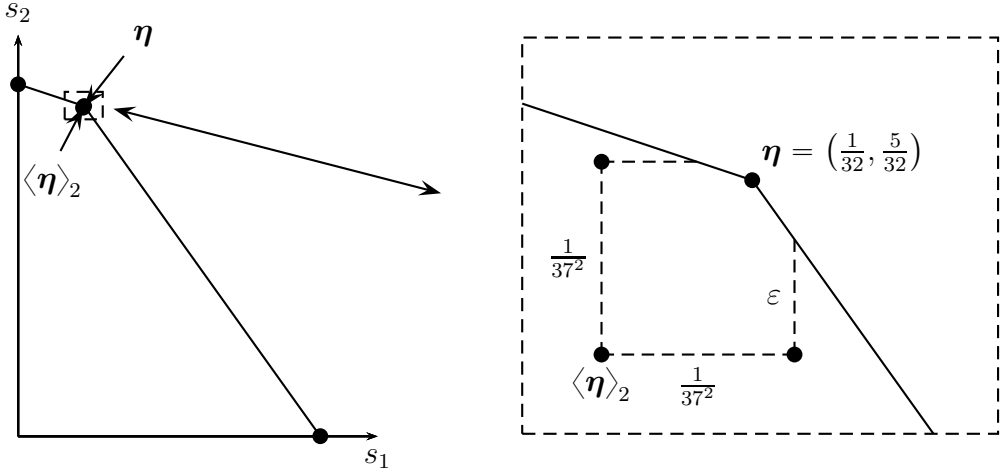
\begin{figure}
\begin{center}
\begin{minipage}[h]{3in}
\psset{xunit=800pt, yunit=800pt}
\begin{pspicture}(0,-.01)(.18,.2)
\psaxes[ticks=none]{->}(0,0)(0,0)(.17,.19)[$s_1$,270][$s_2$,90]


\psline(0,0)(0,.1666)(.03125, .15625)(.14285,0)(0,0)


\psline[linestyle=solid, arrowsize=6pt]{->}(.05,.18)(.03125, .15625)
\uput[45](.05,.18){$\eeta$}

\psline[linestyle=solid, arrowsize=6pt]{->}(.02,.135)(.03067, .1555)
\uput[260](.02,.135){$\tr{\eeta}{2}$}
\psdot[dotsize=6pt](.03067, .1555)

\psdots[dotsize=6pt](0,.1666)(.03125, .15625)(.14285,0)


\psline[linestyle=dashed](.022,.15)(.04,.15)(.04,.163)(.022,.163)(.022,.15)

\psline[arrowsize=6pt]{<->}(.045,.155)(.2,.115)

\end{pspicture}
\end{minipage}
\begin{minipage}[h]{2.2in}
\psset{xunit=10pt, yunit=10pt}
\begin{pspicture}(0,0)(8,9)


\uput[270](0,0){$\tr{\eeta}{2}$}
\psdot[dotsize=6pt](0,0)

\uput[45](5.7, 6.6){$\eeta = \( \frac{1}{32}, \frac{5}{32} \) $}
\psdot[dotsize=6pt](5.7, 6.6)

\psline[linestyle=dashed](-3,-3)(15,-3)(15,12)(-3,12)(-3,-3)

\psline(5.7,6.6)(-3,9.5)
\psline(5.7, 6.6)(12.557,-3)

\psset{linestyle=dashed}
\psline(0,0)(7.3,0)
\uput[270](3.65,0){$\frac{1}{{37}^{2}}$}
\psline(0,0)(0,7.3)
\uput[180](0,3.65){$\frac{1}{{37}^{2}}$}
\psdots[dotsize=6pt](0,7.3)(7.3,0)

\psline(7.3,0)(7.3,4.36) 
\uput[180](7.3,2.18){$\error$}
\psline(0,7.3)(3.6,7.3) 
\end{pspicture}
\end{minipage}

\caption{$p=37, \dee = L = 2$, and $\fpt{f} = \tr{\frac{3}{16}}{2} + \error = .6 \ 34 \ \overline{22 \ 7 \ 14 \ 29} \ (\base 37) $} 
\label{BinomialExample p=37: F}
\end{center}
\end{figure}

\end{Example}


\subsection{Proof of Theorem \ref{BinomialMainTheorem: T}}
The rest of this section is dedicated to the proof of Theorem \ref{BinomialMainTheorem: T}.  We will rely heavily on Lemma \ref{GrandSimplifying: L}, Lemma \ref{reducetosupport: L}, as well as on estimates for $F$-pure thresholds given in \cite[\MTP]{Polynomials}.  We  break up the proof into three parts.  We will continue to use $f$ to denote a polynomial with $\Support(f) = \sM$, and we write $f = u_1 \x^{\a} + u_2 \x^{\b}$.

\begin{Lemma}
\label{reducetosupport: L}
If $\alpha \in \frac{1}{p^e} \cdot \mathbb{N}$, then $f^{p^e \alpha} \equiv \sum \binom{ p^e \alpha }{\gg} \u^{p^e \gg} \x^{p^e \EM \gg} \bmod \bracket{\m}{e}$, where we sum over all $\gg \in \frac{1}{p^e} \cdot \mathbb{N}^2 \cap \LIP$ such that $| \gg | = \alpha$.  In particular, $f^{p^e \alpha} \notin \bracket{\m}{e}$ if and only if there exists $\gg \in \frac{1}{p^e} \cdot \mathbb{N} \cap \LIP$  with $| \gg | = \alpha$  and $\binom{ p^e \alpha}{\gg} \neq 0 \bmod p$.
\end{Lemma}

\begin{proof}  We know that $f^{p^e \alpha} = \sum_{| \k | = p^e \alpha} \binom{p^e \alpha}{\k} \u^{\k} \x^{\EM \k}$.  If $\x^{\EM \k} \notin \bracket{\m}{e}$, then $\EM \k \vl p^e \cdot \vone_m$, so that $\gg = \frac{1}{p^e} \cdot \k \in \LIP$ and $| \gg | = \alpha$.  As $f$ is a binomial, and it is easy to see that there is no gathering of terms when raising $f$ to powers, and thus $f^{p^e \alpha} \notin \bracket{\m}{e}$ if and only if one of the summands given by the binomial theorem is not in $\bracket{\m}{e}$.
\end{proof}

\begin{Remark}
In proving Theorem \ref{BinomialMainTheorem: T}, we will use Lemma \ref{reducetosupport: L} to compute $\max \set{ l : f^l \notin \bracket{\m}{e}}$, and hence $\fpt{f}$, in terms of the geometry of $\P$.  This explains why the statements in Theorem \ref{BinomialMainTheorem: T} do not depend on the coefficients of $f$.
\end{Remark}

\begin{TheoremPt1}
If $L = \infty$, then $\fpt{f} = | \eeta |$.
\end{TheoremPt1}
\begin{proof}
This follows from \cite[\MTP]{Polynomials}.
\end{proof}

Suppose that $L < \infty$, so that $\tr{\eta_1}{L} + \tr{\eta_2}{L} + \frac{1}{p^L} = \tr{\eta_1 + \eta_2}{L}$ by Lemma \ref{Carrying: L}.  This fact is crucial to many of the arguments that follow, and we will apply it without further mention.  In fact, the same statement holds after replacing $L$ by $d$, as defined in Theorem \ref{BinomialMainTheorem: T}.  

\begin{Lemma} 
\label{LastLemma: L}
Suppose that $L < \infty$, and let  $\dee: = \max \set{ e \leq L : \digit{\eta_1}{e} + \digit{\eta_2}{e} \leq p-2}$.  Then $1 \leq d \leq L$.  Furthermore, $\tr{\eta_1}{\dee} + \tr{\eta_2}{\dee} + \frac{1}{p^{\dee}} = \tr{\eta_1 + \eta_2}{L}$.  In particular, $\tr{\eta_1 + \eta_2}{L} \in \frac{1}{p^d} \cdot \mathbb{N}$.
\end{Lemma}

\begin{proof}

If $L=0$, then $\digit{\eta_1}{1} + \digit{\eta_2}{1} \geq p$, which is impossible as $| \eeta | \leq 1$.  Thus, $L \geq 1$.   If $\dee = 0$, then $\digit{\eta_1}{e} + \digit{\eta_2}{e} = p-1$ for $1 \leq e \leq L$, and it follows that $\tr{\eta_1}{L} + \tr{\eta_2}{L} = \sum_{e=1}^L \frac{p-1}{p} = \frac{p^L-1}{p^L}$.  By definition, $\digit{\eta_1}{L+1} + \digit{\eta_2}{L+1} \geq p$, and so \[ \eta_1 + \eta_2 > \tr{\eta_1}{L+1} + \tr{\eta_2}{L+1} = \tr{\eta_1}{L} + \tr{\eta_2}{L} + \frac{ \digit{\eta_1}{L+1} + \digit{\eta_2}{L+1}}{p^{L+1}} \geq \frac{p^L-1}{p^L} + \frac{1}{p^L} = 1,\] which again contradicts the fact that $| \eeta | \leq 1$.  We conclude that $\dee \geq 1$.

For the remainder, we assume $\dee < L$.  Then $\tr{\eta_1}{L} + \tr{\eta_2}{L} = \tr{\eta_1}{\dee} + \tr{\eta_2}{\dee} + \sum_{e={\dee + 1}}^L \frac{p-1}{p^e}$, so $\tr{\eta_1 + \eta_2}{L} = \tr{\eta_1}{L} + \tr{\eta_2}{L} + \frac{1}{p^L} = \tr{\eta_1}{\dee} + \tr{\eta_2}{\dee} + \sum_{e = \dee + 1}^L \frac{p-1}{p^e} + \frac{1}{p^L} = \tr{\eta_1}{\dee} + \tr{\eta_2}{\dee} + \frac{1}{p^{\dee}}$.
\end{proof}

We now continue with our proof of Theorem \ref{BinomialMainTheorem: T}.

\begin{TheoremPt2}
Suppose $L < \infty$.   Then $1 \leq d \leq L$.  Furthermore, if neither $\tr{\eeta}{d} + (\frac{1}{p^d}, 0)$ nor $\tr{\eeta}{d} + (0, \frac{1}{p^d})$ is in $\LIP$, then $\fpt{f} = \tr{\eta_1 + \eta_2}{L}$. 
\end{TheoremPt2}

\begin{proof}
We saw $1 \leq d \leq L$ in Lemma \ref{LastLemma: L}, and it follows from \cite[\MTP]{Polynomials} that 
\begin{equation} \label{pt2one: e} 
\tr{\eta_1 + \eta_2}{L} = \tr{\eta_1}{L} + \tr{\eta_2}{L} + \frac{1}{p^L}   \leq \fpt{f}. \end{equation}  We claim the inequality in \eqref{pt2one: e} is strict.  By means of contradiction, suppose it's not.  Applying Lemma \ref{Truncation: L} shows that \begin{equation} \label{pt2two: e} \tr{\eta_1 + \eta_2}{L} \leq \tr{\fpt{f}}{L} = \max \set{ l : f^l \notin \bracket{\m}{L}},
\end{equation} where we have used Lemma \ref{FPTTruncation: L} to obtain the  equality in \eqref{pt2two: e}.  Apparently, \eqref{pt2two: e} shows that $f^{p^L \tr{\eta_1 + \eta_2}{L}} \notin \bracket{\m}{L}$.  However, Lemma \ref{LastLemma: L} allows us to rewrite this as \[ f^{p^L \tr{\eta_1 + \eta_2}{L}} = \( f^{p^d \tr{\eta_1 + \eta_2}{L}} \)^{p^{L-d}}  \notin \bracket{\m}{L},\] and the flatness of Frobenius then implies that \begin{equation} \label{pt2three: e} f^{p^d \tr{\eta_1 + \eta_2}{L}} \notin \bracket{\m}{d}. \end{equation}

Applying Lemma \ref{reducetosupport: L} to \eqref{pt2three: e} shows there exists an element 
\begin{equation} \label{pt2four: e} \gg = (\gamma_1, \gamma_2) \in \frac{1}{p^d} \cdot \mathbb{N} \cap \LIP \text{ with } | \gg | = \tr{\eta_1 + \eta_2}{L}  = \tr{\eta_1}{d} + \tr{\eta_2}{d} + \frac{1}{p^d},\end{equation}
where we have used Lemma \ref{LastLemma: L} to obtain the last equality in \eqref{pt2four: e}.  We see that it is impossible for both $\gamma_1 \leq \tr{\eta_1}{d}$ and $\gamma_2 \leq \tr{\eta_2}{d}$.  Without loss of generality, we will assume that $\gamma_2 > \tr{\eta_2}{d}$, and as $\gamma_2 \in \frac{1}{p^d} \cdot \mathbb{N}$, we may even assume that $\gamma_2 \geq \tr{\eta_2}{d} + \frac{1}{p^d}$.  Finally, combining this inequality with the properties of $\gg$ recorded in \eqref{pt2four: e}, we see that $\gg$ satisfies the conditions of Lemma \ref{GrandSimplifying: L} (with $\delta = 0$), which implies that $\tr{\eeta}{d} + \( 0, \frac{1}{p^d} \) \in \LIP$, a contradiction.  We conclude that equality holds in \eqref{pt2one: e}, i.e. $\fpt{f}  = \tr{\eta_1 + \eta_2}{L}$.
\end{proof}

\begin{Lemma}
\label{ReallyLast: L}
Suppose that either  $\tr{\eeta}{d} + (\frac{1}{p^d}, 0)$ or $\tr{\eeta}{d} + (0, \frac{1}{p^d})$ is in $\LIP$ and set \[ \error = \max \set{ \delta : \tr{\eeta}{\dee} + \( \frac{1}{p^{\dee}}, \delta \) \text{or } \tr{\eeta}{\dee} + \( \delta, \frac{1}{p^{\dee}} \) \text{is in $\P$}}.\]
Then, $0 < \error \leq \tail{\eta_1 + \eta_2}{L}$, with equality if and only if either $\eta_1$ or $\eta_2$ is in $\frac{1}{p^d} \cdot \mathbb{N}$.   
\end{Lemma}

\begin{proof}
It is obvious that $\error > 0$.  Without loss of generality, we assume that $\tr{\eeta}{d} + ( \error, \frac{1}{p^d}) \in \P$.  By Lemma \ref{LastLemma: L}, this element has coordinate sum $\tr{\eta_1 + \eta_2}{L} + \error$, which must be bounded above by $| \eeta | = \tr{\eta_1 + \eta_2}{L} + \tail{\eta_1 + \eta_2}{d}$, as $\eeta$ is maximal.  We conclude that $\error \leq \tail{\eta_1 + \eta_2}{L}$.  In addition, if we set  $\aalpha:= \tr{\eeta}{d} + \(\tail{\eta_1 + \eta_2}{L}, \frac{1}{p^d} \)$, it follows that 
\begin{equation} \label{LL1: e} \error = \tail{\eta_1 + \eta_2}{L} \iff \aalpha \in \P. \end{equation}
We now wish to derive a more useful expression for $\aalpha$.  By Lemma \ref{LastLemma: L}, 
\begin{align*}
\tail{\eta_1 +\eta_2}{L}  =  \eta_1 + \eta_2 - \tr{\eta_1 + \eta_2}{L}  & =  \eta_1 + \eta_2 - \( \tr{\eta_1}{d} + \tr{\eta_2}{d} + \frac{1}{p^d} \)  \\
 & = \tail{\eta_1}{d} + \tail{\eta_2}{d} - \frac{1}{p^d}.
\end{align*}
Substituting this into \eqref{LL1: e}, we see that \begin{equation} \label{LL2: e} \aalpha = \tr{\eeta}{d} + \( \tail{\eta_1}{d} + \tail{\eta_2}{d} - \frac{1}{p^d}, \frac{1}{p^d} \) = \( \eta_1 + \tail{\eta_2}{d} - \frac{1}{p^d}, \tr{\eta_2}{d} + \frac{1}{p^d} \). \end{equation}  

It follows from \eqref{LL2: e} that $| \aalpha | = | \eeta |$.   As $\eeta$ is the unique maximal point of $\P$, we conclude that $\aalpha \in \P$ if and only if $\aalpha = \eeta$, which by \eqref{LL2: e} happens if and only if $\tail{\eta_2}{d} = \frac{1}{p^d}$ and $\eta_2 = \tr{\eta_2}{d} + \frac{1}{p^d}$, which is clearly redundant and equivalent to the single condition that $\tail{\eta_2}{d} = \frac{1}{p^d}$.  In summary, we have just demonstrated that
\begin{equation} \label{LL3: e} \aalpha \in \P \iff \aalpha = \eeta \iff \tail{\eta_2}{d} = \frac{1}{p^d} \iff \eta_2 \in \frac{1}{p^d} \cdot \mathbb{N},\end{equation} where the last equivalence follows from Lemma \ref{Truncation: L}.  Comparing \eqref{LL3: e} with \eqref{LL1: e}, we see that $\error = \tail{\eta_1 + \eta_2}{L}$ if and only if $\eta_2 \in \frac{1}{p^d}$.
\end{proof}

\begin{TheoremPt3}
Suppose either  $\tr{\eeta}{d} + (\frac{1}{p^d}, 0)$ or $\tr{\eeta}{d} + (0, \frac{1}{p^d})$ is in $\LIP$. 
\begin{enumerate}
\item Then $0 < \error \leq \tail{\norm{\eeta}}{L}$, with equality if and only if either $\eta_1$ or $\eta_2$ is in $\Lattice{p^{\dee}}{}$. \vspace{.05in}
\item Furthermore, we have that $ \frac{1}{p^e} \cdot \max \set{ l : f^l \in \bracket{\m}{e} } = \tr{\eta_1 + \eta_2}{L} + \tr{\error}{e}$ for every $e \geq L$.    In particular, $\fpt{f} = \tr{\eta_1 + \eta_2}{L} + \error$. 
\end{enumerate}
\end{TheoremPt3}

\begin{proof}
The first point is the content of Lemma \ref{ReallyLast: L}.  We assume once and for all that \begin{equation} \label{pt3one: e} \tr{\eeta}{d} + \(0, \frac{1}{p^d} \) \in \LIP \text{ and } \ll:= \tr{\eeta}{d} + \( \error, \frac{1}{p^d} \) \in \P. \end{equation}

By definition of $d$, $\digit{\eta_1}{e} + \digit{\eta_2}{e} \leq 1$ for $1 \leq e < d$ while $\digit{\eta_1}{d} + \digit{\eta_2}{d} + 1 \leq p-1$, and the bound $\error \leq \tail{\eta_1 + \eta_2}{L}$ implies that $\digit{\error}{e} = 0$ for all $1 \leq e \leq L$.  Fix $e \geq L$, and set \[ \ll^e:= \tr{\eeta}{d} + \( \tr{\error}{e}, \frac{1}{p^d} \).\]

The preceding arguments show that the entries of $p^e \cdot \ll^e$ add without carrying (in base $p$), so that $\binom{ p^e | \ll^e | }{p^e \ll^e} \neq 0 \bmod p$.  As $\tr{\error}{e} < \error$, it follows from \eqref{pt3one: e} that $\ll^e \in \LIP$.  By Lemma \ref{LastLemma: L}, $| \ll^e | = \tr{\eta_1 +\eta_2}{L} + \tr{\error}{e}$, and applying Lemma \ref{reducetosupport: L} with $\gg = \ll^e$ shows that \begin{equation} \label{pt3two: e} f^{p^e \tr{\eta_1 + \eta_2}{L}  +p^e \tr{\error}{e}} \notin \bracket{\m}{e} \text{ for } e \geq L .\end{equation}

We will now show that \begin{equation} \label{pt3three: e} f^{p^e \tr{\eta_1+\eta_2}{L} + p^e \tr{\error}{e} + 1} \in \bracket{\m}{e} \text{ for } e \geq L. \end{equation}

Assume the statement in \eqref{pt3three: e} is false.  By Lemma \ref{LastLemma: L}, we may write
\[ \(f^{p^d \tr{\eta_1 + \eta_2}{L}} \)^{p^{e-d}} \cdot f^{p^e \tr{\error}{e} + 1} \notin \bracket{\m}{e}. \]

Choose supporting monomials $\mu_1 \in \Support\( f^{p^d \tr{\eta_1 + \eta_2}{L}} \)$ and $\mu_2 \in \Support \( f^{p^e \tr{\error}{e} + 1} \)$ such that $\mu_1^{p^{e-d}} \mu_2 \notin \bracket{\m}{e}$.  Note that $\mu_1 \notin \bracket{\m}{d}$.  By Lemma \ref{reducetosupport: L}, it follows that 

\begin{equation}
\label{pt3mu2: e} \mu_2 = \x^{p^e \EM \bbeta} \text{ with } \bbeta \in \frac{1}{p^e} \cdot \mathbb{N}^2 \cap \LIP \text{ and } | \bbeta | = \tr{\error}{e} + \frac{1}{p^e} \text{ and }
\end{equation}

\begin{equation} \label{pt3mu1: e}  \mu_1 = \x^{p^d \EM \aalpha}  \text{ with } \aalpha \in \frac{1}{p^d} \cdot \mathbb{N}^2 \cap \LIP \text{ and } | \aalpha | = \tr{\eta_1 + \eta_2}{L} = \tr{\eta_1}{d} + \tr{\eta_2}{d} + \frac{1}{p^d}. \end{equation}

The condition that $\mu_1^{p^{e-d}} \cdot \mu_2 = \x^{p^e \EM \aalpha + p^e \EM \bbeta} \notin \bracket{\m}{e}$ then implies that  
\begin{equation} \label{pt3four: e} 
\gg : = \aalpha + \bbeta \in \LIP.
\end{equation}

It follows from \eqref{pt3mu1: e} that either $\alpha_1 > \tr{\eta_1}{d}$ or $\alpha_2 > \tr{\eta_2}{d}$.  Without loss of generality, we will assume that $\alpha_2 > \tr{\eta_2}{d}$, and as $\alpha_2 \in \frac{1}{p^d} \cdot \mathbb{N}$, we may even assume that $\alpha_2 \geq \tr{\eta_2}{d} + \frac{1}{p^d}$.  It follows that $\gamma_2 \geq \tr{\eta_2}{d} + \frac{1}{p^d}$.  Furthermore, both \eqref{pt3mu2: e} and \eqref{pt3mu1: e} show that $| \gg | = \tr{\eta_1}{d} + \tr{\eta_2}{d} + \frac{1}{p^d} + \( \tr{\error}{e} + \frac{1}{p^e} \)$.  Combining these observations with \eqref{pt3four: e}, we apply Lemma \ref{GrandSimplifying: L} (with $\delta = \tr{\error}{e} + \frac{1}{p^e}$) to obtain that \[ \tr{\eeta}{d} + \( \tr{\error}{e} + \frac{1}{p^e}, \frac{1}{p^d} \) \in \LIP.\]    As $\error \leq \tr{\error}{e} + \frac{1}{p^e}$, this is impossible by the definition of $\error$.  Thus, \eqref{pt3four: e} is also impossible, and we conclude that \eqref{pt3three: e} holds.
\end{proof}

\bibliographystyle{alpha}
\bibliography{refs}

\end{document}